\documentclass[10pt]{amsart}

\usepackage[english]{babel}
\usepackage{amsmath}
\usepackage{amsthm}
\usepackage{enumerate}
\usepackage[top=2.5cm, bottom=2.5cm, left=3.1cm, right=3.1cm]{geometry}
\newtheorem{thm}{Theorem}
\newtheorem{lem}[thm]{Lemma}

\newtheorem{cor}[thm]{Corollary}

\newtheorem{ques}[thm]{Question}
\theoremstyle{remark}

\newcommand{\Q}{\mathbb{Q}}
\newcommand{\vol}{\rm vol}
\newcommand{\inter}{\rm int}

\title{On the characterising slopes of hyperbolic knots}
\author{Duncan McCoy}
\address{Department of Mathematics \\
         The University of Texas At Austin \\
         Austin, TX, 78712, USA}
\email{d.mccoy@math.utexas.edu}
\date{}
\begin{document}

\begin{abstract}
A slope $p/q$ is a characterising slope for a knot $K$ in $S^3$ if the oriented homeomorphism type of $p/q$-surgery on $K$ determines $K$ uniquely. We show that when $K$ is a hyperbolic knot its set of characterising slopes contains all but finitely many slopes $p/q$ with $q \geq 3$. We prove stronger results for hyperbolic $L$-space knots, showing that all but finitely many non-integer slopes are characterising. The proof is obtained by combining Lackenby's proof that for a hyperbolic knot any slope $p/q$ with $q$ sufficiently large is characterising with genus bounds derived from Heegaard Floer homology. 
\end{abstract}
\maketitle

\section{Introduction}
Given a knot $K\subseteq S^3$, we say that $p/q\in \Q$ is a {\em characterising slope} for $K$ if the oriented homeomorphism type of the manifold obtained by $p/q$-surgery on $K$ determines $K$ uniquely. That is $p/q$ is characterising for $K$ if $S^3_K(p/q)\cong S^3_{K'}(p/q)$ for some $K' \subseteq S^3$ implies that $K=K'$\footnote{Throughout the paper, we use $Y'\cong Y$ to denote the existence of a orientation-preserving homeomorphism between $Y$ and $Y'$.}.
In general determining the set of characterising slopes for a given knot is challenging. It was a long-standing conjecture of Gordon, eventually proven by Kronheimer, Mrowka, Ozsv{\'a}th and Szab{\'o}, that every slope is a characterising slope for the unknot \cite{kronheimer07lensspacsurgeries}. Ozsv{\'a}th and Szab{\'o} have also shown that every slope is a characterising slope for the trefoil and the figure-eight knot \cite{Ozsvath2006trefoilsurgery}. Since then, Ni and Zhang -- who introduced the characterising slope terminology -- studied characterising slopes for torus knots, showing that $T_{5,2}$ has only finitely many non-characterising slopes which are not negative integers and exhibiting infinitely many characterising slopes for each torus knot \cite{Ni14characterizing}. It was later shown that any torus knot has only finitely many non-characterising slopes which are not negative integers \cite{mccoy2014torus}. More recently, Lackenby showed that every knot in $S^3$ has infinitely many characterising slopes \cite{Lackenby17characterizing}.

It was conjectured by Ni and Zhang that hyperbolic knots can have only finitely many non-characterising slopes. This was proven to be false by Baker and Motegi who produced examples of knots (including hyperbolic knots) with infinitely many non-characterising slopes \cite{baker16characterizing}. As their examples included only integer slopes as non-characterising slopes, one might wonder about the possibility of non-integer non-characterising slopes.
\begin{ques}[cf. Question 4.4 of \cite{baker16characterizing}]\label{ques:hyp}
Does every hyperbolic knot have only finitely many non-integer non-characterising slopes?
\end{ques}
As evidence for a positive answer to this question Lackenby showed that for a hyperbolic knot $p/q$ is characterising for $K$ whenever $q$ is sufficiently large \cite[Theorem~1.2]{Lackenby17characterizing}. The purpose of this paper is to strengthen this result.
\begin{thm}\label{thm:main}
A hyperbolic knot can have only finitely many non-characterising slopes with $|q|\geq 3$.
\end{thm}
If we add the condition that $K$ is an $L$-space knot\footnote{We say that $K$ is an $L$-space knot if it admits {\em positive} $L$-space surgeries.}, then we can obtain stronger results, answering Question~\ref{ques:hyp} affirmatively, as well as showing many integer slopes are characterising.
\begin{thm}\label{thm:Lspace}
A hyperbolic $L$-space knot has only finitely many non-characterising slopes which are not negative integers. 
\end{thm}
As far the author is aware, this provides the only known examples of integer characterising slopes for any hyperbolic knot other than the figure-eight knot. Since torus knots are known to have only finitely many non-characterising slopes which are not negative integers this immediately yields the following corollary.
\begin{cor}
A non-satellite $L$-space knot has only finitely many non-characterising slopes which are not negative integers.
\end{cor}

In order to show that $p/q$ is characterising for a hyperbolic knot $K$ whenever $q$ is sufficiently large, Lackenby shows that for any other hyperbolic knot $K'$ with $S_{K'}^3(p/q)\cong S_{K}^3(p/q)$ the geometry of $K'$ is sufficiently constrained to ensure that $K=K'$ for large $q$. Key to his argument is that the length of the slope of $p/q$ (as measured on the boundary of a horoball neighbourhood of the cusp in the complement of $K'$) is bounded below by an increasing function of $|q|$ which does not depend on $K'$. 
So in order to adapt Lackenby's approach to find characterising slopes for small $q$ and large $p$, we need to bound the length of $p/q$ below by an increasing function of $|p|$ which does not depend on $K'$. Such a lower bound is obtained by combining a result of Agol which allows us to bound the length of the longitude of a knot in terms of its genus \cite[Theorem~5.1]{Agol00BoundsI} with results derived from Heegaard Floer homology which constrain the genera of two knots with a common surgery. For Theorem~\ref{thm:main} the Heegaard Floer input is hidden in the following theorem.

\begin{thm}\cite[Theorem~1.7]{mccoy2014torus}\label{thm:HFKrecovery}
Let $K,K'\subseteq S^3$ be knots such that $S_{K}^3(p/q)\cong S_{K'}^3(p/q)$. If
\[|p| \geq 12+4q^2 +4qg(K)
\quad \text{and} \quad q\geq 3,\]
then $g(K)=g(K')$.
\end{thm}
The stronger conclusions of Theorem~\ref{thm:Lspace} comes from a corresponding result for $L$-space knots. 

\begin{thm}\cite[Theorem~1.8]{mccoy2014torus}\label{thm:technical2}
Suppose that $K$ is an $L$-space knot. If $S_{K}^3(p/q)\cong S_{K'}^3(p/q)$ for some $K'\subseteq S^3$ and either
\begin{enumerate}[(i)]
\item $p\geq 12+4q^2+4qg(K)$ or
\item $p\leq -(12+4q^2 +2qg(K))$ and $q\geq 2$
\end{enumerate}
holds, then $g(K)=g(K')$ and $K'$ is fibred.
\end{thm}

If one wishes to prove that hyperbolic knots have only finitely many non-integer non-characterising slopes, then it suffices to prove an analogue of Theorem~\ref{thm:HFKrecovery} that applies to half-integer surgeries. Theorem~\ref{thm:HFKrecovery} is proven by calculating the Heegaard Floer homology of $S^3_K(p/q)$ and $S^3_{K'}(p/q)$ using the mapping cone formula and comparing the absolute gradings. Just as Theorem~\ref{thm:HFKrecovery} can be extended to Theorem~\ref{thm:technical2} for $L$-space knots, it is probable that this approach can yield results in the half-integer case for other knots with simple knot Floer homology. However, it seems unlikely to the author that an unconditional statement for half-integer surgeries can be achieved by this approach alone.

\section{The proof}
Given a 3-manifold $M$ with a toroidal boundary component and a slope $\sigma$ on this boundary component, we will use $M(\sigma)$ to denote the Dehn filling along $\sigma$. If $K$ is a knot in $S^3$ we will use $S^3_K$ to denote the knot exterior $S^3\setminus \inter(N(K))$. So $p/q$-surgery on $K$ will be denoted by $S_K^3(p/q)$. A 3-manifold $M$ is hyperbolic if its interior admits a complete finite-volume hyperbolic structure. A knot $K\subseteq S^3$ is hyperbolic, when $S^3_K$ is hyperbolic. 
Recall that Mostow rigidity guarantees that if two hyperbolic 3-manifolds are homeomorphic, then they are isometric. So we may assume that geometric features of a hyperbolic 3-manifold, such as the volume or the shortest closed geodesic, are preserved by homeomorpisms. 

Given a slope $\sigma$ on the boundary of a compact hyperbolic 3-manifold one can assign a length to $\sigma$ by choosing a horoball neighbourhood $N$ of the cusps of $M$. There is a natural Euclidean metric on $\partial N$ and we say the length of $\sigma$ is the length of the shortest curve on $\partial N$ with the slope of $\sigma$. In general, the length of $\sigma$ depends on the choice of horoball neighbourhood. However, if $M$ has only a single cusp, then there is a unique choice of maximal horoball. Given a slope $p/q$ for $S^3_K$, we will use $\ell_K(p/q)$ to denote the length of $p/q$ with respect this maximal horoball neighbourhood.

\subsection{Slope lengths}
The first step is to verify the following proposition, which is a mild reformulation of \cite[Theorem~3.1]{Lackenby17characterizing}. 
\begin{lem}\label{lem:hyp_control}
Let $K\subseteq S^3$ be a hyperbolic knot. There are constants $C_1$ and $C_2$ such that if $K'$ is a hyperbolic knot with $S_{K}^3(p/q)\cong S_{K'}^3(p/q')$ for $\ell_{K'}(p/q')>C_1$ and $|p|+|q|>C_2$, then $K=K'$ and $q=q'$.
\end{lem}
Although a proof is provided for completeness, the reader should note that our proof is exactly the same as Lackenby's except with minor changes to emphasise the role of slope length and a restriction to knots in $S^3$. Three theorems from hyperbolic geometry are needed in  the proof. They are taken largely unchanged from Section 2 of \cite{Lackenby17characterizing}, where further discussion can be found. First, a precise version of Thurston's hyperbolic Dehn surgery theorem is required.  
\begin{thm}\cite[Theorem~2.1]{Lackenby17characterizing}.\label{thm:Thurston_hyp}
Let $M$ be a compact orientable hyperbolic 3-manifold with toroidal boundary components $T_1, \dots, T_n$. Consider a sequence of slopes
$(\sigma_1^i, \dots, \sigma_n^i)$, where $\sigma_j^i$ lies on $T_j$ and $\sigma_j^i\neq \sigma_j^{i'}$ if $i\neq i'$. Then for all sufficiently large $i$, $M(\sigma_1^i,\dots, \sigma_n^i)$ is hyperbolic and the cores of the filling solid tori are geodesics with lengths tending to zero as $i\rightarrow \infty$. Moreover there is $\varepsilon>0$ independent of $i$ such that all other primitive geodesics in $M(\sigma_1^i,\dots, \sigma_n^i)$ have length at least $\varepsilon$. For any horoball neighbourhood $N$ of the cusps of $M$, there is a horoball neighbourhood $N_i$ of the cusps of $M(\sigma_1^i,\dots, \sigma_n^i)$ such that the inclusion
\[ M \setminus N \rightarrow M(\sigma_1^i,\dots, \sigma_n^i)\setminus N_i
\]
is bilipschitz with constant tending to one as $i \rightarrow \infty$
\end{thm}
Secondly, we need to know that for an infinite collection of hyperbolic 3-manifolds of bounded volume, some subsequence of them can be obtained by Dehn filling on another hyperbolic manifold with more cusps. See \cite{ThurstonNotes} or \cite[Theorem E.4.8]{BenedettiHyperbolic}.
\begin{thm}\cite[Theorem~2.2]{Lackenby17characterizing}.\label{thm:magic}
Let $M_i$ be a sequence of distinct oriented hyperbolic 3-manifolds with volume bounded above by $V$. Then there is a hyperbolic 3-manifold $M$ with volume at most $V$ and toroidal boundary components $T_1, \dots, T_n$ such that there is a subsequence of the $M_i$ and a sequence of slopes $(\sigma_1^i, \dots, \sigma_n^i)$ such that
\[M_i= M(\sigma_1^i, \dots, \sigma_n^i)\] and $\sigma_j^i\neq \sigma_j^{i'}$ for $i\neq i'$.
\end{thm}
Finally we need explicit bounds on the volume of hyperbolic 3-manifolds in terms of the length of filling curves. The upper bound is due to Thurston \cite{ThurstonNotes} and the lower bound is due to Futer, Kalfagianni and Purcell \cite[Theorem~1.1]{Futer08filling}.
\begin{thm}\cite[Theorem~2.4]{Lackenby17characterizing}\label{thm:volbound}
Let $S^3_K$ be the complement of a hyperbolic knot in $S^3$. If the slope $p/q$ has length $\ell=\ell_K(p/q)>2\pi$, then $S^3_K(p/q)$ is hyperbolic with volume satisfying
\[
\left(1-\left(\frac{2\pi}{\ell}\right)^2\right)^{3/2} \vol(S^3_K) \leq \vol(S^3_K(p/q))<\vol(S^3_K).
\]
\end{thm}

We are ready to proceed with the proof of Lemma~\ref{lem:hyp_control}. 

\begin{proof}[Proof of Lemma~\ref{lem:hyp_control}] If the constants $C_1$ and $C_2$ do not exist, then there is a sequence of hyperbolic knots $K_i$ with slopes $p_i/q_i$ and $p_i/q_i'$ such that
\begin{enumerate}[(a)]
\item $S_{K}^3(p_i/q_i)\cong S_{K_i}^3(p_i/q_i')$ for all $i$,
\item $\ell_{K_i}(p_i'/q_i')\rightarrow \infty$ and $|p_i|+|q_i|\rightarrow \infty$ as $i\rightarrow \infty$, but
\item for all $i$ we have $K_i\neq K$ or $q_i'\neq q_i$.
\end{enumerate}
We will show that such a sequence results in a contradiction.

First assume the sequence $S_{K_i}^3$ includes infinitely many distinct manifolds. By passing to a subsequence, we may assume that the $S^3_{K_i}$ are all distinct. Since $\ell_{K_i}(p_i/q_i')$ will exceed $4\pi$ for $i$ large enough, Theorem~\ref{thm:volbound} show that 
$\left(\frac{3}{4}\right)^{3/2}\vol(S^3_{K_i})\leq \vol(S_{K_i}^3(p_i/q_i'))$ for $i$ sufficiently large. However $\vol(S_{K_i}^3(p_i/q_i'))=\vol(S_{K}^3(p_i/q_i))$ is bounded above by $\vol(S_{K}^3)$. This gives the upper bound $\vol(S_{K_i}^3)\leq (\frac{3}{4})^{3/2}\vol(S_K^3)$ for all sufficiently large $i$. Thus we see that there is some $V$ such that $\vol(S_{K_i}^3)<V$ for all $i$.

By Theorem~\ref{thm:magic} we can pass to a further subsequence and assume that there is a hyperbolic $M$ of finite volume with toroidal boundary components $T_1, \dots, T_n, T_{n+1}$ and a sequence of slopes $(\sigma_1^i, \dots, \sigma_n^i)$ on $T_1, \dots, T_n$ such that $\sigma_j^i\neq \sigma_j^{i'}$ for $i\neq i'$ and $S^3_{K_i} \cong M(\sigma_1^i, \dots, \sigma_n^i, \star)$, where $\star$ denotes that we are leaving $T_{n+1}$ unfilled.
As the knot complements $S_{K_i}^3$ are distinct manifolds, we have $n\geq 1$. We may consider the slope $p_i/q_i'$ as slope a $\sigma_i$ on $T_{n+1}$. Thus, we get a sequence of slopes $\sigma_i$ such that $S_{K_i}^3(p_i/q_i')\cong M(\sigma_1^i, \dots, \sigma_n^i,\sigma^i)$ for all $i$.

Let $N$ be a horoball neighbourhood of the cusps of $M$. By Theorem~\ref{thm:Thurston_hyp}, in each $S^3_{K_i}$ there is a horoball neighbourhood $N_i$ of the cusp such that the inclusion $M\setminus N\rightarrow S^3_{K_i}\setminus N_i$ is bilipschitz with constant approaching one. Thus since $\ell_{K_i}(p_i/q_i')\rightarrow \infty$ the length of $\sigma_i$ as measured in $\partial N$ must also tend to infinity. In particular, by taking a further subsequence we can assume the slopes $\sigma_i$ are distinct.

Therefore Theorem~\ref{thm:Thurston_hyp} shows that the cores of the filling solid tori in $M(\sigma_1^i, \dots, \sigma_n^i,\sigma^i)$ are geodesics of length tending to zero. Thus for any $\varepsilon>0$, $M(\sigma_1^i, \dots, \sigma_n^i,\sigma^i)$ contains at least $n+1\geq 2$ closed geodesics of length less than $\varepsilon$ when $i$ sufficiently large. However, Theorem~\ref{thm:Thurston_hyp} also shows that there is $\delta>0$, such that for $i$ sufficiently large the core of $S_{K}^3(p_i/q_i)$ is the only geodesic of length less than $\delta$. This is clearly a contradiction.

Thus we can assume that $S_{K_i}^3$ include only finitely many distinct manifolds. By passing to a subsequence if necessary, we can further assume that there is some $K'\subseteq S^3$ such that $S_{K_i}^3\cong S_{K'}^3$ for all $i$.

We may assume that the homemorphisms $S_{K}^3(p_i/q_i)\rightarrow S_{K'}^3(p_i/q_i')$ map the shortest closed geodesic in $S_{K}^3(p_i/q_i)$ to shortest geodesic in $S_{K'}^3(p_i/q_i')$. However for $i$ sufficiently large Theorem~\ref{thm:Thurston_hyp} shows this shortest geodesic is the core of the filling solid tori in both manifolds. Thus the homeomorphism restricts to give a homeomorphism of knot complements $S_{K'}^3\cong S_{K}^3$. By the knot complement theorem this shows that $K=K'$ and that the meridian of $K$ is mapped to the meridian of $K'$ \cite{Gordon89complement}. Since the homeomorphism must also map null-homologous curves to null-homologous curves it must also preserve longitudes, showing that we also have $p_i/q_i'=p_i/q_i$. This contradicts the initial assumptions on the $p_i/q_i'$ and $K_i$.
\end{proof}
The following lemma provides the slope length bounds required to apply Lemma~\ref{lem:hyp_control}.
\begin{lem}\label{lem:length_bound}
Let $K\subseteq S^3$ be a hyperbolic knot of genus $g$. Then 
\[\ell_K(p/q) \geq \frac{\sqrt{3}\,|q|}{6}\quad\text{and}\quad \ell_K(p/q) \geq \frac{\sqrt{3}\,|p|}{6(2g-1)}\]
\end{lem}
\begin{proof}
By considering the area of a cusp, Cooper and Lackenby show that \cite[Lemma~2.1]{Cooper98negatively}
\[\ell_K(\alpha)\geq \frac{\sqrt{3}\,\Delta(\alpha,\beta)}{\ell_K(\beta)},\]
where $\alpha$ and $\beta$ are any two slopes on the boundary of $K$ and $\Delta(\alpha,\beta)$ denotes their distance (cf. \cite[Lemma~8.1]{Agol00BoundsI}).
 Since $\Delta(1/0,p/q)=|q|$ and $\ell_K(1/0)\leq 6$ by the 6-theorem \cite{Agol00BoundsI, Lackenby00-6thm}, this gives the bound on $\ell_K(p/q)$ in terms of $|q|$. Since $\Delta(0/1,p/q)=|p|$ and $\ell_K(0/1)\leq 6(2g-1)$ by \cite[Theorem~5.1]{Agol00BoundsI}, this also gives the bound on $\ell_K(p/q)$ in terms of $|p|$.
\end{proof}

\subsection{Hyperbolic surgeries on satellite knots}
We also need to understand when non-hyperbolic knots can have hyperbolic surgeries.
\begin{lem}\label{lem:satellite}
Suppose that $K'$ is a satellite knot with $S_{p/q}^3(K')$ hyperbolic. Then there is a hyperbolic knot $K''$ with $S_{p/q}^3(K')\cong S_{p/q'}^3(K'')$ for some $q' > q$. Moreover if $q\geq 2$ or $K'$ is fibred, then $g(K'') \leq g(K')$
\end{lem}
\begin{proof}
Let $T$ be an incompressible torus in $S^3\setminus K'$. We can consider $K'$ as a knot in the solid torus $V$ bounded by $T$. Thus we can consider $K'$ as a satellite with companion given by the core $K''$ of $V$. By choosing $T$ to ensure that $S^3\setminus K'$ contains no further incompressible tori, we can assume that $K''$ is not a satellite knot. Hence by the work of Thurston $K'$ is a torus knot or a hyperbolic knot \cite{Thurston823DKleinanGroups}. Since $S_{K'}^3(p/q)$ is hyperbolic, it is atoroidal and irreducible. Consequently the Dehn filling when considered as a surgery on $V$ must produce a $S^1\times D^2$. However, Gabai has classified knots in $S^1\times D^2$ with non-trivial $S^1\times D^2$ surgeries, showing that $K'$ is either a torus knot or a 1-bridge braid in $V$ \cite{gabai89solidtori}. Moreover since $S^1\times D^2$ fillings on 1-bridge braids only occur for integer surgery slopes, $K'$ is a cable of $K''$ unless $q=1$. In either event, we have that
\[ S_{K'}^3(p/q) \cong S_{K''}^3(p/q'),\]
for some $q'$. However, it is known that $q'=qw^2$, where $w>1$ is the winding number of $K'$ in $V$ \cite[Lemma~3.3]{Gordon83Satellite}.
Since $S_{K'}^3(p/q)$ is hyperbolic, $K''$ cannot be a torus knot, therefore $K''$ is hyperbolic as required.

If $K'$ is fibred, then $K''$ must also be fibred \cite{Hirasawa2008Fiber}, thus the inequality $g(K'')\leq g(K')$, follows by considering the degrees of their Alexander polynomials. If $q\geq 2$, then $K'$ is a cable of $K''$. It is known that $g(K')\geq  g(K'')$ in this case \cite{Shibuya89genus}. 
\end{proof}
\subsection{Proof of Theorem~\ref{thm:main} and Theorem~\ref{thm:Lspace}}
Let $K$ be a hyperbolic knot. Suppose that $S_{K}^3(p/q)\cong S_{K'}^3(p/q)$ for some $K'$ and some $p/q$ and that one of the following two conditions hold:
\begin{enumerate}
\item $q\geq 3$ or
\item $K$ is an $L$-space knot and $q\geq 2$ or $q=1$ and $p>0$.
\end{enumerate}
Assume also that $|p|+|q|$ is large enough to guarantee $S_{p/q}^3(K)$ is hyperbolic. Since torus knots have no hyperbolic surgeries, Thurston's work shows that $K'$ is either a hyperbolic knot or a satellite knot \cite{Thurston823DKleinanGroups}. If $K'$ is a satellite, then Lemma~\ref{lem:satellite} shows there is a hyperbolic knot $K''$ with $S^3_{K''}(p/q')\cong S^3_{K}(p/q)$ for some $q'>q$ and $g(K'')\leq g(K')$. In either event, there is a hyperbolic knot $L$ with $g(L)\leq g(K')$ such that $S^3_{L}(p/q')\cong S^3_{K}(p/q)$ for $q'\geq q$. Thus to show that $p/q$ is a characterising slope for $K$ it suffices to show that the only possibility is $q=q'$ and $L=K$. 

Lemma~\ref{lem:length_bound} shows that $\ell_L(p/q')\geq \frac{\sqrt{3}\,|q|}{6}$. So Lemma~\ref{lem:hyp_control} shows that there is $C$ such that $q>C$ implies that $K=L$ and $q'=q$. This implies that $p/q$ is a characterising slope if $q>C$. Thus we assume from now on that $q\leq C$. However by Theorem~\ref{thm:HFKrecovery} and Theorem~\ref{thm:technical2}, if we assume further that $|p|\geq 12 + 4C^2 + 4Cg(K)$, then $g(K)=g(K')\geq g(L)$. Therefore for $|p|$ large enough Lemma~\ref{lem:length_bound} shows that $\ell_L(p/q')\geq \frac{\sqrt{3}\,|p|}{6(2g(K)-1)}$. Hence by taking $|p|$ even larger if necessary, Lemma~\ref{lem:hyp_control} applies again to show that $q=q'$ and $K=L$, as required. This concludes the proof and the paper.

\subsection*{Acknowledgements} The author would like to thank Ahmad Issa and Effie Kalfagianni for comments on earlier versions of this paper.

\bibliographystyle{abbrv}
\bibliography{char_bib}
\end{document}